\documentclass[12pt]{article}%
\usepackage{graphicx}
\usepackage{amsmath}
\usepackage{amsfonts}
\usepackage{amssymb}%
\usepackage{amsthm}

\providecommand{\U}[1]{\protect\rule{.1in}{.1in}}
\newtheorem{theorem}{Theorem}[section]

\newtheorem{lemma}[theorem]{Lemma}

\newtheorem{proposition}[theorem]{Proposition}

\newtheorem*{quote1}{Theorem \ref{thm1}}

\begin{document}

\author{Lawrence Reeves\\School of Mathematics and Statistics\\University of Melbourne\\Victoria 3010, Australia.\\email: lreeves@unimelb.edu.au
\and Peter Scott\\Mathematics Department\\University of Michigan\\Ann Arbor, Michigan 48109, USA.\\email: pscott@umich.edu
\and Gadde A.~Swarup\\718 High Street Road\\Glen Waverley\\Victoria 3150, Australia.\\email: anandaswarupg@gmail.com }
\title{A deformation theorem for Poincar\'{e} duality pairs in dimension $3$}
\maketitle

\begin{abstract}
We prove the analogue of Johannson's Deformation Theorem for $PD3$ pairs.

\end{abstract}
\date{}

\begin{center}
\textit{Dedicated to Walter Neumann on his 75th birthday}
\end{center}

\section{Introduction}

The object of this paper is to prove an algebraic analogue of Johannson's
Deformation Theorem \cite{Johannson}, a classical result in the theory of
$3$--manifolds. We recall that if $M$ is an orientable Haken $3$--manifold $M$
with incompressible boundary, then the JSJ decomposition of $M$ is given by a
possibly disconnected (and possibly empty) compact submanifold $V(M)$ which is
called the characteristic submanifold of $M$, such that the frontier of $V(M)$
consists of essential annuli and tori, and each component of $V(M)$ is an
$I$--bundle or a Seifert fibre space. Further each component of the closure of
$M-V(M)$ is simple. The existence and uniqueness of this decomposition is due
to Jaco and Shalen \cite{JacoShalen}, and independently Johannson
\cite{Johannson}. In \cite{Johannson}, Johannson went on to prove his
Deformation Theorem which asserts that if $M$ and $N$ are orientable Haken
$3$--manifolds with incompressible boundary, and if $f:M\rightarrow N$ is a
homotopy equivalence, then $f$ can be homotoped so as to map $V(M)$ to $V(N)$
by a homotopy equivalence, and to map the closure of $M-V(M)$ to the closure
of $N-V(N)$ by a homeomorphism.

In order to formulate our algebraic generalization, we need to introduce a
number of concepts.

If $M$ is an orientable Haken $3$--manifold with incompressible boundary, then
$V(M)$ determines a bipartite graph of groups structure $\Gamma_{M}$ for
$G=\pi_{1}(M)$ which is dual to the frontier $frV(M)$ of $V(M)$. The edges of
$\Gamma_{M}$ correspond to the components of $frV(M)$, and the vertices
correspond to the components of $M$ cut along $frV(M)$. Thus all the edge
groups in $\Gamma_{M}$ are free abelian of rank $1$ or $2$.

In \cite{SS03}, as corrected in \cite{SS03errata}, the authors obtained
canonical graph of groups decompositions for almost finitely presented groups
analogous to the above graph of groups structure $\Gamma_{M}$, determined by
the JSJ decomposition of a $3$--manifold $M$. In particular, for many almost
finitely presented groups $G$, and any integer $n\geq1$, they defined a
bipartite graph of groups decomposition $\Gamma_{n,n+1}(G)$, and showed that
when $G$ is the fundamental group of an orientable Haken $3$--manifold $M$
with incompressible boundary, then $\Gamma_{1,2}(G)$ is almost the same as
$\Gamma_{M}$. The differences, which are small but crucial, are discussed in
detail in \cite{SS05}. It is simple to construct the completion $\Gamma
_{n,n+1}^{c}(G)$ of $\Gamma_{n,n+1}(G)$, and then $\Gamma_{1,2}^{c}(G)$ is
equal to $\Gamma_{M}$. Further details are discussed in \cite{SS05}. Note that
as $\Gamma_{1,2}^{c}(G)$ depends only on $G$, it follows that the JSJ
decomposition of a $3$--manifold $M$ is determined by $\pi_{1}(M)$, even
though non-homeomorphic $3$--manifolds may have isomorphic fundamental groups.

A Poincar\'{e} duality pair is the algebraic analogue of an aspherical
manifold with aspherical boundary components whose fundamental groups inject.
It consists of a group $G$ which corresponds to the fundamental group of the
manifold, and a family $\partial G$ of subgroups which correspond to the
fundamental groups of the boundary components, and the whole setup satisfies
an appropriate version of Poincar\'{e} duality. If the fundamental class lies
in $H_{n+2}(G,\partial G)$, we say that the pair is of dimension $n+2$, and
that $(G,\partial G)$ is a $PD(n+2)$ pair. If $M$ is an orientable Haken
$3$--manifold with incompressible boundary, and if $G$ denotes the fundamental
group of $M$, and $\partial G$ denotes the family of fundamental groups of
components of $\partial M$, then the pair $(G,\partial G)$ is an orientable
$PD3$ pair. Thus for an orientable $PD3$ pair $(G,\partial G)$, the
decomposition $\Gamma_{1,2}^{c}(G)$ is our analogue of the JSJ decomposition
of a $3$--manifold.

In \cite{SS05}, the authors considered the structure of the completion
$\Gamma_{n,n+1}^{c}(G)$ of $\Gamma_{n,n+1}(G)$, in the case of orientable
$PD(n+2)$ pairs $(G,\partial G)$, where $n\geq1$. The results obtained were
very closely analogous to the above description for $3$--manifolds. Thus, for
an orientable $PD(n+2)$ pair $(G,\partial G)$, the algebraic analogue of the
JSJ decomposition of a $3$--manifold is the bipartite graph of groups
$\Gamma_{n,n+1}^{c}(G)$. For brevity, we will denote $\Gamma_{n,n+1}^{c}(G)$
by $\Gamma_{G}$ in this paper. This graph of groups has vertices of two types
denoted $V_{0}$ and $V_{1}$. The $V_{0}$--vertices of $\Gamma_{G}$ correspond
to the components of the characteristic submanifold $V(M)$ of a $3$--manifold
$M$, and the $V_{1}$--vertices of $\Gamma_{G}$ correspond to components of
$M-V(M)$. The edges of $\Gamma_{G}$ correspond to the annuli and tori which
form the frontier of $V(M)$ in $M$. For each vertex $v$ of $\Gamma_{G}$, the
edges of the decomposition incident to $v$ determine a family of subgroups of
$G(v)$ which is denoted by $\partial_{1}v$. The authors of \cite{SS05} also
showed that the decomposition of $\partial G$ induced from $\Gamma_{G}$
determines a family of subgroups of $G(v)$ denoted $\partial_{0}v$. Each of
the groups in $\partial_{0}v$ and $\partial_{1}v$ has a natural structure as a
$PD(n+1)$ pair, and $\partial v=\partial_{0}v\cup\partial_{1}v$ naturally has
the structure of a family of $PD(n+1)$ groups obtained by gluing the pairs in
$\partial_{0}v$ and $\partial_{1}v$ along their boundaries. Thus $\partial v$
forms a sort of boundary of $v$, but since the groups in $\partial v$ may not
inject into $G(v)$, we do not have a notion of Poincar\'{e} duality for the
pair $(G(v),\partial v)$. In the case when $n=1$, a $PD2$ pair is a surface
with boundary \cite{Muller}, and $PD2$ groups are surface groups
\cite{EckmannLinnell}\cite{EckmannMuller}. Thus, in the case of a
$3$--dimensional Poincar\'{e} duality pair $(G,\partial G)$, the families
$\partial_{0}v$ and $\partial_{1}v$ consist of surfaces with boundary and
closed surfaces, and the family $\partial v$ consists of closed surfaces, for
each vertex $v$ of $\Gamma_{G}$. In addition, we showed in \cite{RSS} that if
$v$ is a $V_{0}$--vertex of $\Gamma_{G}$ then the pair $(G(v),\partial v)$ is
a $3$--dimensional manifold. 

There is now a long history of algebraic analogues of the JSJ decomposition
(see \cite{SS03} and \cite{SS05} for a discussion), but there has been much
less discussion of the algebraic significance of Johannson's Deformation
Theorem \cite{Johannson}. Johannson proved his Deformation Theorem using his
theory of boundary patterns. A few years later, Swarup \cite{Sw} showed that
in the special case when $M$ and $N$ admit no essential map of an annulus,
there is a much simpler proof of the Deformation Theorem. His argument used
the purely algebraic concept of the number of ends of a pair of groups. In
\cite{Jaco} (see \cite{Jaco2} for an account with more detail), Jaco used
Swarup's idea to give an alternative proof of Johannson's Deformation Theorem.
Several years later, Swarup's idea was extended by Scott and Swarup in
\cite{SS02} to give another proof of the Deformation Theorem. The arguments in
this paper derive from those used in \cite{SS02}, but the proofs here are
different and essentially self-contained. One reason for the differences is
the lack of a theory of Poincar\'{e} duality for pairs $(G,\partial G)$ in
which the boundary groups do not inject into $G$.

In this paper, we will prove the following algebraic analogue of Johannson's
Deformation Theorem \cite{Johannson} for $PD3$ pairs.

\begin{quote1}
Let $(G,\partial G)$ and $(H,\partial H)$ be two $PD3$ pairs with
$G$ isomorphic to $H$, and let $\Gamma_{G}$ and $\Gamma_{H}$ be the
corresponding isomorphic bipartite graphs of groups. If $v$ in $\Gamma_{G}$
and $w$ in $\Gamma_{H}$ are corresponding $V_{1}$--vertices, then the
isomorphism carries $\partial_{1}v$ to $\partial_{1}w$, and $\partial_{0}v$ to
$\partial_{0}w$, and $\partial v$ isomorphically to $\partial w$.
\end{quote1}

That $\partial_{1}v$ is carried to $\partial_{1}w$ follows from the fact that
$\Gamma_{G}$ depends only on $G$, and not on $\partial G$, as was shown in
\cite{SS05}. The new element is that $\partial_{0}v$ is carried to
$\partial_{0}w$. Note that Johannson's Deformation Theorem \cite{Johannson}
contains more information than this in the $3$--manifold setting. For his
result asserts that if $M$ and $N$ are orientable Haken $3$--manifolds with
incompressible boundary, and if $f:M\rightarrow N$ is a homotopy equivalence,
then $f$ can be homotoped so as to map $V(M)$ to $V(N)$ by a homotopy
equivalence, and to map the closure of $M-V(M)$ to the closure of $N-V(N)$ by
a homeomorphism. The algebraic analogue of Theorem \ref{thm1} asserts only
that $f$ is a homeomorphism from the boundary of the closure of $M-V(M)$ to
the boundary of the closure of $N-V(N)$. The approach taken here to prove
Theorem \ref{thm1} yields yet another proof of Johannson's Deformation
Theorem. In section \ref{applnsto3-manifolds}, we sketch this and give a
detailed comparison with the arguments in \cite{SS02}.

\section{Proof of the Main result\label{section:proofofthemainresult}}

For $PD3$ pairs $(G,\partial G)$ and $(H,\partial H)$, the decompositions
$\Gamma_{G}$ and $\Gamma_{H}$ induce decompositions of $\partial G$ and
$\partial H$ giving rise to $\partial_{0}v$ and $\partial_{0}w$. These
decompositions are described in detail in \cite{SS05} using $K(\pi,1)$ spaces
representing the group pairs. We take spaces $(M,\partial M)$ and $(N,\partial
N)$ to represent the above pairs and start with a split homotopy equivalence
$f:M\rightarrow N$ with inverse $g:N\rightarrow M$. As the edge spaces are now
$2$--dimensional annuli and tori, we can homotop $f$ and $g$ to be
homeomorphisms on the edge spaces. For $i=0,1$, we denote by $V_{i}$ the
subspace of $M$ which is the union of all $V_{i}$--vertex spaces, and let
$W_{0}$ and $W_{1}$ denote the corresponding subspaces of $N$. Thus
$M=V_{0}\cup V_{1}$, and $N=W_{0}\cup W_{1}$, and $f(V_{0})\subset W_{0}$,
$f(V_{1})\subset W_{1}$, and $f$ is a homeomorphism on the edge spaces
$V_{0}\cap V_{1}$.

A crucial role in our arguments is played by the fact that an essential
annulus in a $PD3$ pair $(G,\partial G)$ is enclosed by a $V_{0}$--vertex of
$\Gamma_{G}$. In particular, if an essential annulus in $(G,\partial G)$ is
enclosed by a $V_{1}$--vertex of $\Gamma_{G}$, it must be homotopic into an
edge annulus of that vertex.

Now we establish some notation to be used throughout this section. We fix
corresponding $V_{1}$--vertices $v$ and $w$ of $\Gamma_{G}$ and $\Gamma_{H}$
respectively, let $X$ be the component of $V_{1}$ corresponding to $v$, and
let $Y$ be the component of $W_{1}$, corresponding to $w$. Thus $f$ carries
$X$ to $Y$ with $\partial_{1}X$ going homeomorphically to $\partial_{1}Y$,
where $\partial_{1}X$ and $\partial_{1}Y$ denote the unions of the annuli and
tori which form the edge surfaces of $X$ and $Y$ respectively. Also
$\partial_{0}X$ and $\partial_{0}Y$ denote the unions of the surfaces
$X\cap\partial M$ and $Y\cap\partial N$ respectively. Finally $\partial X$
denotes the union $\partial_{0}X\cup\partial_{1}X$, and $\partial Y$ denotes
the union $\partial_{0}Y\cup\partial_{1}Y$.

For each component $\overline{t}$ of $\partial_{0}X$, we want to deform
$\overline{t}$ into $\partial_{0}Y$, and then obtain a homeomorphism from
$\partial X$ to $\partial Y$. Let $M_{t}$ denote the cover of $M$
corresponding to $\pi_{1}(\bar{t})$, and let $N_{t}$ denote the corresponding
cover of $N$. Denote by $t$ the lift of $\overline{t}$ into $M_{t}$, and by
$X_{t}$ the component of the pre-image of $X$ in $M_{t}$ which contains $t$.
Let $Y_{t}$ denote the corresponding component of the pre-image of $Y$ in
$N_{t}$. Let $\partial_{0}X_{t}$ denote the pre-image in $X_{t}$ of
$\partial_{0}X$, and let $\partial_{1}X_{t}$ denote the pre-image in $X_{t}$
of $\partial_{1}X$. Define $\partial_{0}Y_{t}$ and $\partial_{1}Y_{t}$ in the
same way. Finally let $f_{t}$ denote the induced homotopy equivalence from
$M_{t}$ to $N_{t}$.

Let $A$ be an annulus component of $\partial_{1}X$, and let $B$ be the
corresponding annulus component of $\partial_{1}Y$. We are assuming that $f$
maps $A$ to $B$ by a homeomorphism inducing the given isomorphism on the
corresponding edge groups of $\Gamma_{G}$ and $\Gamma_{H}$. But there are two
possible isotopy classes for this homeomorphism, and we may need to alter the
initial choice. We define a \textit{"flip" on }$A$ to be a map from $M$ to
itself which preserves each of $V_{0}$ and $V_{1}$, is the identity outside
some small neighbourhood of $A$, and whose restriction to $A$ is a
homeomorphism which interchanges the components of $\partial A$. Further a
flip on $A$ is homotopic to the identity map of $M$ by a homotopy supported on
a small neighbourhood of $A$. Thus by composing $f$ with a flip on $A$, we can
homotop $f$ to change the initial choice of isotopy class of homeomorphism
from $A$ to $B$, and this homotopy is supported on a small neighbourhood of
$A$.

We start by considering some special cases.

\begin{lemma}
\label{casewhentisanannulus}Suppose that one of the following holds:

\begin{enumerate}
\item A component $\overline{t}$ of $\partial_{0}X$ is an annulus.

\item There is a component $\overline{t}$ of $\partial_{0}X$ such that
$\overline{t}$ is not closed, and $\pi_{1}(\bar{t})$ has finite index in
$\pi_{1}(X)$.
\end{enumerate}

Then $v$ and $w$ are isolated $V_{1}$--vertices of $\Gamma_{G}$ and
$\Gamma_{H}$ respectively. Further we can homotop $f:M\rightarrow N$ to
arrange that it maps $\partial_{1}X$ to $\partial_{1}Y$ by a homeomorphism,
and maps $\partial_{0}X$ to $\partial_{0}Y$ by a homeomorphism. Thus $f$ also
maps $\partial X$ to $\partial Y$ by a homeomorphism.
\end{lemma}

\begin{proof}
1) As $v$ is a $V_{1}$--vertex of $\Gamma_{G}=\Gamma_{n,n+1}^{c}(G)$, and
$\overline{t}$ is an annulus in $X$ with boundary in $\partial_{1}X$, we can
apply Proposition 2.5 of \cite{RSS}. This tells us that either $\partial
\overline{t}$ lies in a single component $s$ of $\partial_{1}X$, and that
$\overline{t}$ is homotopic into $s$ fixing $\partial\overline{t}$, or that
$v$ is isolated. The first case would imply that $s$ is homotopic into
$\overline{t}$ fixing $\partial s$, and so is homotopic into $\partial M$
fixing $\partial s$, which contradicts the fact that $s$ is an essential
annulus in $(M,\partial M)$. It follows that $v$ is an isolated $V_{1}%
$--vertex of $\Gamma_{G}$, as required. In particular, $\pi_{1}(X)$ is
infinite cyclic. Hence $\pi_{1}(Y)$ is also infinite cyclic, so that any
component of $\partial_{0}Y$ must be an annulus. Now applying Proposition 2.5
of \cite{RSS} again shows that $Y$ is an isolated $V_{1}$--vertex of
$\Gamma_{H}$, as required. Thus by flipping on one of the annuli in
$\partial_{1}X$ if needed, we can homotop $f$ to map $\partial_{1}X$ to
$\partial_{1}Y$ by a homeomorphism, and simultaneously map $\partial_{0}X$ to
$\partial_{0}Y$ by a homeomorphism. This completes the proof of part 1) of the lemma.

2) We will show that some component of $\partial_{0}X$ is an annulus, so that
the result follows from part 1). We use the above notation. As $\pi_{1}%
(\bar{t})$ has finite index in $\pi_{1}(X)$, it follows that $X_{t}$ is
compact. In particular, $\partial_{0}X_{t}$ is also compact. As the inclusion
of $t$ in $X_{t}$ is a homotopy equivalence, there is a retraction $\rho$ of
$X_{t}$ to $t$.

If $\partial_{0}X_{t}$ equals $t$, and $A$ is an annulus component of
$\partial_{1}X_{t}$, the retraction $\rho$ induces a map from $A$ to $t$ which
sends $\partial A$ by a homeomorphism to two components of $\partial t$. Hence
the map from $A$ to $t$ is a proper map of degree $1$, and so induces a
surjection $\pi_{1}(A)\rightarrow\pi_{1}(t)$. This implies that $t$ is an
annulus, as required.

If $\partial_{0}X_{t}$ is not equal to $t$, we let $s$ be another component of
$\partial_{0}X_{t}$. Any loop $\lambda$ in $s$ is homotopic in $X_{t}$ into
$t$, and so determines an annulus in $X_{t}$. If $s$ is not an annulus, we can
choose $\lambda$ to be an essential non-peripheral loop in $s$. It follows
that the annulus in $X_{t}$ determined by $\lambda$ is $\pi_{1}$--injective,
and cannot be properly homotoped in $X_{t}$ into $\partial_{1}X_{t}$, nor into
$\partial_{0}X_{t}$, while keeping its boundary in $\partial_{0}X_{t}$. It
follows that this annulus is essential in $(M,\partial M)$ and cannot be
homotoped into $\partial_{1}X_{t}$, which contradicts the fact that any
essential annulus in a $PD3$ pair $(G,\partial G)$ is enclosed by a $V_{0}%
$--vertex of $\Gamma_{G}$. This contradiction shows that $s$ must be an
annulus, which completes the proof of part 2) of the lemma.
\end{proof}

Next we consider the case in which $\overline{t}$ is closed. This case was
considered in \cite{Sw} and the proof here is similar (see also \cite{KK}).

\begin{lemma}
\label{casewhentisclosed}Using the above notation, if $\overline{t}$ is a
component of $\partial_{0}X$ which is a closed surface, then we can homotop
$f$ to arrange that $f$ maps $\overline{t}$ to a component of $\partial_{0}Y$
by a homeomorphism.
\end{lemma}

\begin{proof}
Again we use the above notation. We claim that each component of $\partial
M_{t}-t$ is contractible. For suppose a component $r$ of $\partial M_{t}-t$ is
not contractible. Then there is an annulus $A$ from $r$ to $t$, since $\pi
_{1}(t)\rightarrow\pi_{1}(M_{t})$ is an isomorphism. As $r$ and $t$ are
distinct components of $\partial M_{t}$, the annulus $A$ is essential in
$M_{t}$, and so must be properly homotopic to an annulus in $\partial_{1}%
X_{t}$. As $t$ is closed, this is a contradiction which proves the claim.

The long exact homology sequence for the pair $(M_{t},\partial M_{t})$ with
integer coefficients yields the exact sequence
\[
H_{2}(\partial M_{t})\rightarrow H_{2}(M_{t})\rightarrow H_{2}(M_{t},\partial
M_{t})\rightarrow H_{1}(\partial M_{t})\rightarrow H_{1}(M_{t}).
\]
As the inclusion of $t$ into $M_{t}$ is a homotopy equivalence, and the other
components of $\partial M_{t}$ are contractible, it follows that the first and
last maps in this sequence are isomorphisms. Also $H_{3}(M_{t})=H_{3}(t)=0$.
It follows that $H_{3}(M_{t},\partial M_{t})=0$ and $H_{2}(M_{t},\partial
M_{t})=0$. The first equality tells us that $M_{t}$ is not compact, and the
second implies, by duality, that $H_{c}^{1}(M_{t})=0$. Since $f_{t}$ is a
proper homotopy equivalence from $M_{t}$ to $N_{t}$, it follows that $N_{t}$
is not compact and $H_{c}^{1}(N_{t})=0$. Hence $H_{3}(N_{t},\partial N_{t}%
)=0$, and by duality $H_{2}(N_{t},\partial N_{t})=0$. As $H_{2}(N_{t}%
)\cong\mathbb{Z}$, the long exact homology sequence of the pair $(N_{t}%
,\partial N_{t})$ shows that $H_{2}(\partial N_{t})\cong\mathbb{Z}$. Thus
there is exactly one closed component $s$ of $\partial N_{t}$, and the induced
map $H_{2}(s)\rightarrow H_{2}(N_{t})$ is an isomorphism. Now let $\rho$
denote a retraction $M_{t}\rightarrow t$, and consider the composite map
$s\subset N_{t}\overset{g_{t}}{\rightarrow}M_{t}\overset{\rho}{\rightarrow}t$.
Since each of these three maps induces an isomorphism on $H_{2}$, the
composite map $s\rightarrow t$ has degree $1$. Thus the induced map $\pi
_{1}(s)\rightarrow\pi_{1}(t)$ is surjective and hence an isomorphism. Thus $s$
is a retract of $N_{t}$, and $f_{t}\mid t:t\rightarrow N_{t}$ can be deformed
into $s$.

We next want to show that $s$ is actually in $Y_{t}$. Since each component of
$\partial M_{t}-t$ is contractible, there are no annuli in $\partial_{1}X_{t}%
$, and so no annuli in $\partial_{1}Y_{t}$. If $s$ is not in $Y_{t}$, the fact
that $s$ is homotopic into $Y_{t}$ implies that there must be a
non-contractible component of $\partial_{1}Y_{t}$ which can only be a torus.
Denote this torus by $T$, and note that as $s$ is closed, $\pi_{1}(s)$ must be
of finite index in $\pi_{1}(T)$. Consider the images $\overline{s}$ and
$\overline{T}$ in $N$ of $s$ and $T$ respectively. If we cut $N$ along
$\overline{T}$, it follows from \cite{BE01} that we get one or two $PD3$
pairs, depending on whether $\overline{T}$ separates $N$. Let $(K,\partial K)$
denote the pair such that $\partial K$ contains $\bar{s}$. Then $\partial K$
also contains one or two copies of $\overline{T}$, depending on whether
$\overline{T}$ separates $N$. Since $\bar{s}$ and $\overline{T}$ carry
commensurable subgroups of $\pi_{1}(Y)$, it follows from Lemma 2.2 of
\cite{KR2} that $(K,\partial K)$ must be trivial, meaning that $\partial K$
consists of two copies of $K$. Thus $\partial K$ consists entirely of one copy
of $\bar{s}$ and one copy of $\overline{T}$, and each carries $K$. But this
implies that $\overline{T}$ separates $N$ and splits $\pi_{1}(N)$ trivially,
which is a contradiction. It follows that $s$ must lie in $Y_{t}$, and hence
that $t$ can be deformed into $s$ staying in $Y_{t}$. Thus we can homotop $f$
to arrange that $f$ maps $\overline{t}$ to a component of $\partial_{0}Y$ by a
homeomorphism, as required.
\end{proof}

The main part of the proof of Theorem \ref{thm1} will be the remaining cases
in which $\overline{t}$ has boundary, which we handle with a sequence of
propositions. By Lemma \ref{casewhentisanannulus}, we can assume that no
component of $\partial_{0}X$ is an annulus, and that $\pi_{1}(\bar{t})$ has
infinite index in $\pi_{1}(X)$. Using our previous notation, the cover $X_{t}$
of $X$ contains $t$, and $X_{t}$ and $M_{t}$ are non-compact.

\begin{proposition}
\label{prop0}Using the above notation, let $\overline{t}$ be a component of
$\partial_{0}X$ with non-empty boundary, such that $\pi_{1}(\bar{t})$ has
infinite index in $\pi_{1}(X)$, and suppose that no component of $\partial
_{0}X$ is an annulus. Then the following hold.

\begin{enumerate}
\item Each component of $\partial_{1}X_{t}$ which covers an annulus component
of $\partial_{1}X$ is either an annulus meeting $t$ in a single boundary
component, or is contractible.

\item Let $C$ be a non-contractible component of $\partial_{0}X_{t}$, other
than $t$. Then there is an annulus component $A$ of $\partial_{1}X_{t}$, such
that one component of $\partial A$ is contained in $C$, and the other
component of $\partial A$ is contained in $t$. Further $\pi_{1}(C)$ is
infinite cyclic.

\item Each component of $\partial_{1}X_{t}$ which covers a torus component of
$\partial_{1}X$ is contractible.
\end{enumerate}
\end{proposition}

\begin{proof}
1) Suppose that $C$ is an annulus component of $\partial_{1}X_{t}$ which does
not meet $t$. There is a $\pi_{1}$--injective annulus $A$ in $X_{t}$ joining a
component of $\partial C$ to $t$. As this component lies in a component $s$ of
$\partial_{0}X_{t}$, the annulus $A$ is in $(M_{t},\partial M_{t})$. As $s$
and $t$ are distinct components of $\partial_{0}X_{t}$, it follows that $A$ is
essential in $(M_{t},\partial M_{t})$, and so must be homotopic into
$\partial_{1}X_{t}$ while keeping $\partial A$ in $\partial_{0}X_{t}$. But
this implies that either $C$ meets $t$, or that $s$ is an annulus, either of
which contradicts our hypotheses. We conclude that any annulus component of
$\partial_{1}X_{t}$ must meet $t$.

Let $A$ be a component of $\partial_{1}X_{t}$ which meets $t$. Thus $A$ is an
annulus. Suppose that $\partial A$ is contained in $t$. There is a retraction
of $X_{t}$ to $t$, and it maps $A$ to $t$ sending $\partial A$ by a
homeomorphism to two components of $\partial t$. Hence this map $A\rightarrow
t$ has degree $1$, which implies that $t$ is an annulus. This contradiction
shows that each component of $\partial_{1}X_{t}$ which meets $t$ is a compact
annulus which meets $t$ in exactly one boundary component. This completes the
proof of part 1).

2) As $C$ is non-contractible, there is a $\pi_{1}$--injective annulus joining
$C$ to $t$. This cannot be homotopic into $\partial_{0}X_{t}$ while keeping
its boundary in $\partial_{0}X_{t}$, and so it must be homotopic into
$\partial_{1}X_{t}$ while keeping its boundary in $\partial_{0}X_{t}$. This
implies that there is an annulus component $A$ of $\partial_{1}X_{t}$, such
that one component of $\partial A$ is contained in $C$, and the other
component of $\partial A$ is contained in $t$. Further this argument shows
that any loop in $C$ is homotopic into $\partial C$, showing that $\pi_{1}(C)$
must be infinite cyclic, as required.

3) Suppose there is a non-contractible component $C$ of $\partial_{1}X_{t}$
which covers a torus component $T$ of $\partial_{1}X$, and let $H$ be an
infinite cyclic subgroup of $\pi_{1}(C)$. As $\pi_{1}(T)$ normalises $H$, it
follows that $H$ is a subgroup of $\pi_{1}(t)$ which has infinite index in its
normalizer. Consider the cover $M_{H}$ of $M$, with $\pi_{1}(M_{H})=H$, and
let $t_{H}$ denote the component of the pre-image of $t$ with $\pi_{1}%
(t_{H})=H$. As $H$ has infinite index in its normalizer, $M_{H}$ has
infinitely many components of $\partial M_{H}$ which contain translates of
$t_{H}$ and have fundamental group $H$. It follows that there are infinitely
many distinct annuli in $(M_{H},\partial M_{H})$, all carrying $H$, so that
there are crossing such annuli. As these annuli are all enclosed by the
$V_{1}$--vertex $v$ of $\Gamma_{H}$, corresponding to $X$, this is a
contradiction, which completes the proof of part 3).
\end{proof}

Now let $A_{1},\dots,A_{n}$ denote the annuli of $\partial_{1}X_{t}$ which
meet $t$, and let $\partial A_{i}=\{a_{1},a_{i}^{\prime}\}$ with $a_{i}$ in
$t$. We have corresponding annuli $B_{1},\dots,B_{n}$ in $\partial_{1}Y_{t}$
with $\partial B_{i}=\{b_{i},b_{i}^{\prime}\}$. We may assume that $f_{t}$
carries $A_{i}$ homeomorphically to $B_{i}$ with $a_{i}$ going to $b_{i}$
initially. Note that in the covering projections $p_{t}:M_{t}\rightarrow M$,
$q_{t}:N_{t}\rightarrow N$ some of these annuli may be identified. An annulus
$A$ in $\partial_{1}X$ which meets $\overline{t}$ lifts to two annuli in
$\partial_{1}X_{t}$ if $\partial A\subset\partial\bar{t}$, and lifts to one
annulus otherwise.

Since the cover $p_{t}:M_{t}\rightarrow M$ is formed with respect to the image
of $\pi_{1}(t)$, we have a clearer picture of the cover $M_{t}$ than of
$N_{t}$. We can use this to obtain some information about the homology of
$X_{t}$, as follows. Each annulus $A_{i}$ lies in a component, say $P_{i}$, of
the closure of $M_{t}-X_{t}$. As the inclusion of $X_{t}$ in $M_{t}$ is a
homotopy equivalence, it follows that $X_{t}$ meets $P_{i}$ only in $A_{i}$,
and that the inclusion of $A_{i}$ into $P_{i}$ is a homotopy equivalence. Note
that $\partial_{0}X_{t}$ equals the intersection $\partial M_{t}\cap X_{t}$.
We let $\partial_{0}P_{i}$ denote the intersection $\partial M_{t}\cap P_{i}$.
Proposition \ref{prop0} tells us that if $\Theta$ is a component of
$\partial_{1}X_{t}$ other than the $A_{i}$'s, then $\Theta$ is contractible.
As the inclusion of $X_{t}$ in $M_{t}$ is a homotopy equivalence, it follows
that $\Theta$ lies in a component, say $P_{\Theta}$, of the closure of
$M_{t}-X_{t}$, that $X_{t}$ meets $P_{\Theta}$ only in $\Theta$, and that
$P_{\Theta}$ is contractible. We have that $M_{t}$ is the union of $X_{t}$,
the $P_{i}$'s, and the $P_{\Theta}$'s, and that $\partial M_{t}$ is the union
of $\partial_{0}X_{t}$, the $\partial_{0}P_{i}$'s, and the $\partial
_{0}P_{\Theta}$'s.

\begin{proposition}
\label{prop1} Using the above notation, let $\overline{t}$ be a component of
$\partial_{0}X$ with non-empty boundary, such that $\pi_{1}(\bar{t})$ has
infinite index in $\pi_{1}(X)$, and suppose that no component of $\partial
_{0}X$ is an annulus. Then $H_{2}(X_{t},\partial_{0}X_{t})\cong\mathbb{Z}^{n}%
$, and is freely generated by $[A_{1}],\dots,[A_{n}]$.
\end{proposition}

\begin{proof}
Recall that $\partial_{0}X_{t}$ consists of $t$, various components containing
some $a_{i}^{\prime}$, and perhaps some other components. Proposition
\ref{prop0} tells us that these extra components are contractible, and that
any component of $\partial_{0}X_{t}$ which contains some $a_{i}^{\prime}$ has
infinite cyclic fundamental group. The long exact homology sequence of the
pair $(X_{t},\partial_{0}X_{t})$ implies that
\[
0\rightarrow H_{2}(X_{t},\partial_{0}X_{t})\rightarrow H_{1}(\partial_{0}%
X_{t})\rightarrow H_{1}(X_{t})\rightarrow0
\]
is exact, as $H_{2}(X_{t})=H_{2}(t)=0$, and $H_{1}(\partial_{0}X_{t})\cong
H_{1}(t)\oplus\sum_{i=1}^{n}H_{1}(a_{i}^{\prime})$, and $H_{1}(t)$ maps
isomorphically onto $H_{1}(X_{t})$. Since $a_{i}$ and $a_{i}^{\prime}$ map to
the same elements in $H_{1}(X_{t})$, the proposition follows.
\end{proof}

Next we prove the following.

\begin{proposition}
\label{prop2} Using the above notation, let $\overline{t}$ be a component of
$\partial_{0}X$ with non-empty boundary, such that $\pi_{1}(\bar{t})$ has
infinite index in $\pi_{1}(X)$, and suppose that no component of $\partial
_{0}X$ is an annulus. Then $H_{2}(X_{t},\partial_{0}X_{t})\rightarrow
H_{2}(M_{t},\partial M_{t})$ is an isomorphism and both are freely generated
by $[A_{1}],\dots,[A_{n}]$.
\end{proposition}

\begin{proof}
Similar arguments to those in Proposition \ref{prop0} tell us that the
components of $\partial_{0}P_{i}$ which do not contain $a_{i}$ or
$a_{i}^{\prime}$ are contractible, and the other components of $\partial
_{0}P_{i}$ have infinite cyclic fundamental group. Also, for each contractible
component $\Theta$ of $\partial_{1}X_{t}$, all components of $\partial
_{0}P_{\Theta}$ are contractible. It follows that the inclusion of $A_{i}$
into $P_{i}$ induces an isomorphism from $H_{1}(\partial A_{i})$ to
$H_{1}(\partial_{0}P_{i})$ and an injection from $H_{0}(\partial A_{i})$ to
$H_{0}(\partial_{0}P_{i})$, and similar statements hold when $\Theta$ is the
universal cover of an annulus. If $\Theta$ is the universal cover of a torus,
we note that $H_{1}(\partial_{0}P_{\Theta})$ is zero. As the inclusion of
$A_{i}$ into $P_{i}$ is a homotopy equivalence, it follows from the long exact
homology sequences of the pairs $(A_{i},\partial A_{i})$ and $(P_{i}%
,\partial_{0}P_{i})$ that $H_{2}(A_{i},\partial A_{i})\rightarrow H_{2}%
(P_{i},\partial_{0}P_{i})$ is an isomorphism and $H_{1}(A_{i},\partial
A_{i})\rightarrow H_{1}(P_{i},\partial_{0}P_{i})$ is an injection. If $\Theta$
is the universal cover of an annulus or torus, then $H_{2}(\Theta
,\partial\Theta)=H_{2}(P_{\Theta},\partial_{0}P_{\Theta})=0$, and
$H_{1}(\Theta,\partial\Theta)\rightarrow H_{1}(P_{\Theta},\partial
_{0}P_{\Theta})$ is an injection. Now we consider the Mayer-Vietoris sequence
for the pair $(M_{t},\partial M_{t})$ expressed as the union of $(X_{t}%
,\partial_{0}X_{t})$ and $(\cup P_{i},\cup\partial_{0}P_{i})\cup(\cup
P_{\Theta},\cup\partial_{0}P_{\Theta})$. We obtain the short exact sequence%

\begin{gather*}
0\rightarrow%
{\displaystyle\sum\limits_{i=1}^{n}}
H_{2}(A_{i},\partial A_{i})\oplus%
{\displaystyle\sum\limits_{\Theta}}
H_{2}(\Theta,\partial\Theta)\rightarrow\\%
{\displaystyle\sum\limits_{i=1}^{n}}
H_{2}(P_{i},\partial_{0}P_{i})\oplus%
{\displaystyle\sum\limits_{\Theta}}
H_{2}(P_{\Theta},\partial_{0}P_{\Theta})\oplus H_{2}(X_{t},\partial_{0}%
X_{t})\rightarrow H_{2}(M_{t},\partial M_{t})\rightarrow0,
\end{gather*}

where the first term is $H_{3}(M_{t},\partial M_{t})$ which is zero as $M_{t}$
is not compact. The final term being zero reflects the fact that the boundary
map from $H_{2}(M_{t},\partial M_{t})$ is zero, as the next map in the
Mayer-Vietoris sequence is injective.

As $H_{2}(\Theta,\partial\Theta)=H_{2}(P_{\Theta},\partial_{0}P_{\Theta})=0$,
and $H_{2}(A_{i},\partial A_{i})\rightarrow H_{2}(P_{i},\partial_{0}P_{i})$ is
an isomorphism, it follows that $H_{2}(X_{t},\partial_{0}X_{t})\rightarrow
H_{2}(M_{t},\partial M_{t})$ is an isomorphism, as required.
\end{proof}

Next we want to apply the same arguments to $N_{t}$, $Y_{t}$ and the
components of the closure of $N_{t}-Y_{t}$ to obtain the analogous isomorphism
but without the information about the generators. For $Y_{t}$ we have only the following.

\begin{proposition}
\label{prop3} Using the above notation, let $\overline{t}$ be a component of
$\partial_{0}X$ with non-empty boundary, such that $\pi_{1}(\bar{t})$ has
infinite index in $\pi_{1}(X)$, and suppose that no component of $\partial
_{0}X$ is an annulus. Then $H_{2}(Y_{t},\partial_{0}Y_{t})\rightarrow
H_{2}(N_{t},\partial N_{t})$ is an isomorphism.
\end{proposition}

\begin{proof}
Proposition \ref{prop0} shows that the components of $\partial_{1}X_{t}$
consist of the annuli $A_{1},\ldots A_{n}$ together with contractible
components. As we have a graph of groups isomorphism between $\Gamma_{G}$ and
$\Gamma_{H}$, it follows that the components of $\partial_{1}Y_{t}$ consist of
the annuli $B_{1},\ldots B_{n}$ together with contractible components. Each
annulus $B_{i}$ lies in a component, say $Q_{i}$, of the closure of
$N_{t}-Y_{t}$, and $Y_{t}$ meets $Q_{i}$ only in $B_{i}$. If $\Theta$ is one
of these contractible components of $\partial_{1}Y_{t}$, then $\Theta$ lies in
a component, say $Q_{\Theta}$, of the closure of $N_{t}-Y_{t}$, such that
$Y_{t}$ meets $Q_{\Theta}$ only in $\Theta$, and $Q_{\Theta}$ is contractible.
We have that $N_{t}$ is the union of $Y_{t}$, the $Q_{i}$'s, and the
$Q_{\Theta}$'s, and that $\partial N_{t}$ is the union of $\partial_{0}Y_{t}$,
the $\partial_{0}Q_{i}$'s, and the $\partial_{0}Q_{\Theta}$'s. Further, as in
the proof of Proposition \ref{prop2}, the components of $\partial_{0}Q_{i}$
which do not contain $b_{i}$ or $b_{i}^{\prime}$ are contractible, and the
other components of $\partial_{0}Q_{i}$ have infinite cyclic fundamental
group. Also, for each contractible component $\Theta$ of $\partial_{1}Y_{t}$,
all components of $\partial_{0}Q_{\Theta}$ are contractible. As the inclusion
of $B_{i}$ into $Q_{i}$ is a homotopy equivalence, it follows from the long
exact homology sequences of the pairs $(B_{i},\partial B_{i})$ and
$(Q_{i},\partial_{0}Q_{i})$ that $H_{2}(B_{i},\partial B_{i})\rightarrow
H_{2}(Q_{i},\partial_{0}Q_{i})$ is an isomorphism and $H_{1}(B_{i},\partial
B_{i})\rightarrow H_{1}(Q_{i},\partial_{0}Q_{i})$ is an injection. If $\Theta$
is the universal cover of an annulus or torus, then $H_{2}(\Theta
,\partial\Theta)=H_{2}(Q_{\Theta},\partial_{0}Q_{\Theta})=0$, and
$H_{1}(\Theta,\partial\Theta)\rightarrow H_{1}(Q_{\Theta},\partial
_{0}Q_{\Theta})$ is an injection. Now as in the proof of Proposition
\ref{prop2}, we apply the Mayer-Vietoris sequence for the pair $(N_{t}%
,\partial N_{t})$ expressed as the union of $(Y_{t},\partial_{0}Y_{t})$ and
$(\cup Q_{i},\cup\partial_{0}Q_{i})\cup(\cup Q_{\Theta},\cup\partial
_{0}Q_{\Theta})$, to deduce the result. But note that although this argument
is very similar to the proof of Proposition \ref{prop2}, the groups involved
may be very different, as $\partial M_{t}$ and $\partial N_{t}$ may be different.
\end{proof}

However duality implies that $H_{2}(M_{t},\partial M_{t})\cong H_{c}^{1}%
(M_{t})\cong H_{c}^{1}(N_{t})\cong H_{2}(N_{t},\partial N_{t})$. Thus

\begin{proposition}
\label{prop4} All the groups in Propositions \ref{prop2} and \ref{prop3} are
free abelian of rank $n$.
\end{proposition}

In the $3$--manifold setting, the map $H_{2}(Y_{t},\partial_{0}Y_{t}%
)\rightarrow H_{2}(Y_{t},\partial Y_{t})$ is zero, and later we will be able
to show this holds in the present setting. But we will begin by proving
something weaker in Proposition \ref{prop7} below.

We write $A$ for the union of the $A_{i}$'s, and $B$ for the union of the
$B_{i}$'s. Note that we have homotopy equivalences of pairs $(M_{t}%
,A)\rightarrow(N_{t},B)$ and $(M_{t},\partial A)\rightarrow(N_{t},\partial B)$.

\begin{proposition}
\label{prop5} Using the above notation, let $\overline{t}$ be a component of
$\partial_{0}X$ with non-empty boundary, such that $\pi_{1}(\bar{t})$ has
infinite index in $\pi_{1}(X)$, and suppose that no component of $\partial
_{0}X$ is an annulus. Then $H_{2}(M_{t},\partial A)$ is free abelian of rank
$n+1$, generated by $[t]$ and $[A_{1}],\dots,[A_{n}]$.
\end{proposition}

\begin{proof}
Consider the long exact homology sequence of the triple $(M_{t},\partial
M_{t},\partial A)$. As $M_{t}$ is not compact, we have $H_{3}(M_{t},\partial
M_{t})=0$. Also Proposition \ref{prop2} implies that the boundary map
$H_{2}(M_{t},\partial M_{t})\rightarrow H_{1}(\partial M_{t},\partial A)$ is
zero. Thus we obtain the short exact sequence%
\[
0\rightarrow H_{2}(\partial M_{t},\partial A)\rightarrow H_{2}(M_{t},\partial
A)\rightarrow H_{2}(M_{t},\partial M_{t})\rightarrow0.
\]

Recall that $\partial M_{t}$ is the union of $\partial_{0}X_{t}$,
$\cup\partial_{0}P_{i}$ and $\cup\partial_{0}P_{\Theta}$. Further the proofs
of Propositions \ref{prop1} and \ref{prop2} show that each component of
$\partial_{0}X_{t}$, $\cup\partial_{0}P_{i}$ and $\cup\partial_{0}P_{\Theta}$
other than $t$, is either contractible and disjoint from $t$, or meets $t$ in
a single boundary component and has infinite cyclic fundamental group. It
follows that $H_{2}(\partial M_{t},\partial A)$ is infinite cyclic generated
by $[t]$. Now the above short exact sequence implies the result of the proposition.
\end{proof}

The homotopy equivalence $(M_{t},\partial A)\rightarrow(N_{t},\partial B)$
immediately implies the following result.

\begin{proposition}
\label{prop6} Using the above notation, let $\overline{t}$ be a component of
$\partial_{0}X$ with non-empty boundary, such that $\pi_{1}(\bar{t})$ has
infinite index in $\pi_{1}(X)$, and suppose that no component of $\partial
_{0}X$ is an annulus. Then $H_{2}(N_{t},\partial B)$ is free abelian of rank
$n+1$, generated by $[f(t)]$ and $[B_{1}],\dots,[B_{n}]$.
\end{proposition}

Now we are able to show the following result.

\begin{proposition}
\label{prop7} Using the above notation, let $\overline{t}$ be a component of
$\partial_{0}X$ with non-empty boundary, such that $\pi_{1}(\bar{t})$ has
infinite index in $\pi_{1}(X)$, and suppose that no component of $\partial
_{0}X$ is an annulus. Then the map $H_{2}(Y_{t},\partial_{0}Y_{t})\rightarrow
H_{2}(Y_{t},\partial Y_{t})$ sends $[f(t)]$ to zero.
\end{proposition}

\begin{proof}
Consider the long exact homology sequence of the triple $(N_{t},\partial
N_{t},\partial B)$. As $N_{t}$ is not compact, we have $H_{3}(N_{t},\partial
N_{t})=0$, as in the proof of Proposition \ref{prop4}. Thus we obtain the
exact sequence%
\[
0\rightarrow H_{2}(\partial N_{t},\partial B)\rightarrow H_{2}(N_{t},\partial
B)\rightarrow H_{2}(N_{t},\partial N_{t}).
\]

We do not know that the boundary map $H_{2}(N_{t},\partial N_{t})\rightarrow
H_{1}(\partial N_{t},\partial B)$ is zero, nor do we know the rank of
$H_{2}(\partial N_{t},\partial B)$. However Proposition \ref{prop4} tells us
that $H_{2}(N_{t},\partial N_{t})$ is free abelian of rank $n$. Thus the image
of the map $H_{2}(N_{t},\partial B)\rightarrow H_{2}(N_{t},\partial N_{t})$
has rank at most $n$. In particular, the generators $[f(t)]$ and
$[B_{1}],\dots,[B_{n}]$ of $H_{2}(N_{t},\partial B)$ are mapped to dependent
elements of $H_{2}(N_{t},\partial N_{t})$. Now Proposition \ref{prop3} tells
us that $H_{2}(Y_{t},\partial_{0}Y_{t})\rightarrow H_{2}(N_{t},\partial
N_{t})$ is an isomorphism. Thus the elements $[f(t)]$ and $[B_{1}%
],\dots,[B_{n}]$ are dependent elements of $H_{2}(Y_{t},\partial_{0}Y_{t})$.
As the map $H_{2}(Y_{t},\partial_{0}Y_{t})\rightarrow H_{2}(Y_{t},\partial
Y_{t})$ sends each $[B_{i}]$ to zero, it must also send $[f(t)]$ to zero, as required.
\end{proof}

Now we come to the key argument.

\begin{proposition}
\label{prop8}Using the above notation, let $\overline{t}$ be a component of
$\partial_{0}X$ with non-empty boundary, such that $\pi_{1}(\bar{t})$ has
infinite index in $\pi_{1}(X)$, and suppose that no component of $\partial
_{0}X$ is an annulus. Then there is a component $S$ of $\partial_{0}Y_{t}$
whose boundary is contained in $\partial B$, and contains exactly one
component from each $\partial B_{i}$, such that $\pi_{1}(S)\rightarrow\pi
_{1}(N_{t})$ is an isomorphism.
\end{proposition}

\begin{proof}
As $\partial t$ is the union of the $a_{i}$'s, it follows that $f(\partial t)$
is the union of the $b_{i}$'s. Thus Proposition \ref{prop7} implies that there
is a (possibly disconnected) compact surface $S$ in $\partial Y_{t}$ whose
boundary is the union of the $b_{i}$'s. If $S$ is not contained in
$\partial_{0}Y_{t}$, it must contain some $B_{i}$. By replacing $S$ by the
closure of $S-B_{i}$, we can replace $S$ by a new (possibly disconnected)
compact surface in $\partial Y_{t}$ whose boundary is contained in $\partial
B$, and contains exactly one component from each $\partial B_{i}$. By
repeating this process as needed, we will eventually find a (possibly
disconnected) compact surface $S$ in $\partial_{0}Y_{t}$ whose boundary is
contained in $\partial B$, and contains exactly one component from each
$\partial B_{i}$.

Let $s$ be a component of $S$ and consider the composite map $s\overset{g_{t}%
}{\rightarrow}X_{t}\overset{\rho}{\rightarrow}t$, where $\rho$ is a retraction
of $X_{t}$ to $t$. The resulting map is $\pi_{1}$--injective, carries each
boundary component of $s$ to a boundary component of $t$ by a homeomorphism,
and sends distinct components of $\partial s$ to distinct components of
$\partial t$. It follows that the composite map $s\rightarrow t$ has degree
$1$. Hence it is onto on $\pi_{1}$, and so an isomorphism on $\pi_{1}$. It
follows that this map $s\rightarrow t$ is properly homotopic to a
homeomorphism, so that $s$ must be equal to $S$. Hence $S$ satisfies the
conclusion of the proposition, as required.
\end{proof}

Now we are ready to complete the proof of Theorem \ref{thm1}.

\begin{theorem}
\label{thm1}Let $(G,\partial G)$ and $(H,\partial H)$ be two $PD3$ pairs with
$G$ isomorphic to $H$, and let $\Gamma_{G}$ and $\Gamma_{H}$ be the
corresponding isomorphic bipartite graphs of groups. If $v$ in $\Gamma_{G}$
and $w$ in $\Gamma_{H}$ are corresponding $V_{1}$--vertices, then the
isomorphism carries $\partial_{1}v$ to $\partial_{1}w$, and $\partial_{0}v$ to
$\partial_{0}w$, and $\partial v$ isomorphically to $\partial w$.
\end{theorem}

\begin{proof}
If there is a component $\overline{t}$ of $\partial_{0}X$, which is an
annulus, or if there is a component $\overline{t}$ of $\partial_{0}X$ such
that $\overline{t}$ is not closed, and $\pi_{1}(\bar{t})$ has finite index in
$\pi_{1}(X)$, then the result was proved in Lemma \ref{casewhentisanannulus}.

Now we consider the remaining cases. We need to show that for each component
$\overline{t}$ of $\partial_{0}X$, we can homotop $f:M\rightarrow N$ to
arrange that $f$ maps $\overline{t}$ by a homeomorphism to a component of
$\partial_{0}X$, and that distinct components of $\partial_{0}X$ are sent to
distinct components of $\partial_{0}Y$.

If $\overline{t}$ is a component of $\partial_{0}X$, which is a closed
surface, then Lemma \ref{casewhentisclosed} shows that we can homotop
$f:M\rightarrow N$ to arrange that $f$ maps $\overline{t}$ by a homeomorphism
to a component of $\partial_{0}X$.

Now suppose that no component of $\partial_{0}X$ is an annulus, and let
$\overline{t}$ be a component of $\partial_{0}X$, which has boundary, such
that $\pi_{1}(\bar{t})$ has infinite index in $\pi_{1}(X)$. Proposition
\ref{prop8} produces a component $S$ of $\partial_{0}Y_{t}$ such that
$\partial S$ consists of one and exactly one from each pair $\{b_{j}%
,b_{j}^{\prime}\}$, and $\pi_{1}(S)\rightarrow\pi_{1}(N_{t})$ is an
isomorphism. Let $\overline{S}$ denote the image in $\partial_{0}Y$ of $S$.

Under our indexing, $f(a_{i})=b_{i}$ and $\partial S$ contains only one of
$\{b_{i},b_{i}^{\prime}\}$. Thus by deforming along $B_{i}$ if necessary, we
can push $f_{t}(t)$ homeomorphically onto $S$. We will use flips to alter $f$
to arrange that, after applying these flips, $f(\overline{t})$ can be deformed
into $\overline{S}$ while fixing $\partial\overline{t}$. By repeating for all
components of $\partial_{0}X$, we will arrange that, after applying certain
flips, $f(\partial X)$ can be deformed to a homeomorphism from $\partial X$ to
$\partial Y$, while fixing $\partial_{1}X$.

There are two cases to consider, depending on whether two of the $A_{i}$ have
the same image in $M$, under $p_{t}$.

\textit{Case 1:} $p_{t}(A_{1})=p_{t}(A_{2})$, and so $q_{t}(B_{1})=q_{t}%
(B_{2})$.

Recall that in our notation, $f_{t}(a_{1})=b_{1}$ and $f_{t}(a_{2})=b_{2}$.
Thus under these identifications, we have $p_{t}(a_{1}^{\prime})=p_{t}(a_{2}%
)$, $p_{t}(a_{1})=p_{t}(a_{2}^{\prime})$, $q_{t}(b_{1}^{\prime})=q_{t}(b_{2}%
)$, $q_{t}(b_{1})=q_{t}(b_{2}^{\prime})$. Moreover, $\partial S$ contains one
of $\{b_{1},b_{1}^{\prime}\}$ and one of $\{b_{2},b_{2}^{\prime}\}$.

\textit{Case 1a:} $\partial S$ contains $b_{1}=f_{t}(a_{1})$.

We claim that in this case $b_{2}=f_{t}(a_{2})$ is in $\partial S$. If not
$\partial S$ contains $b_{1}$ and $b_{2}^{\prime}$ which are identified under
$q_{t}$. Thus, the image of $\partial X_{t}$ contains the image of $S$ with
$b_{1}$ and $b_{2}^{\prime}$ identified and also the image of $B_{1}$ and
$B_{2}$. Thus we have a branched surface at $q_{t}(b_{1})$ with three
branches, whereas $\partial X=\partial_{0}X\cup\partial_{1}X$ is a closed
surface. Therefore, if $\partial S$ contains $b_{1}$, it also contains $b_{2}$
and so we can homotop $f_{t}(t)$ into $S$ fixing $\partial t$. Hence we can
homotop $f(\overline{t})$ into the image of $S$ in $N$ fixing $\partial
\overline{t}$, as required.

\textit{Case 1b:} $\partial S$ contains $b_{1}^{\prime}=f_{t}(a_{1}^{\prime})$.

Arguing as in case 1a), it follows that $\partial S$ must consist of
$b_{1}^{\prime}$ and $b_{2}^{\prime}$. We also have $f_{t}(a_{1})=b_{1}$ and
$f_{t}(a_{2})=b_{2}$. Thus it is not possible to homotop $f_{t}(t)$ into $S$
fixing $\partial t$. However, if we compose $f$ with simultaneous flips on
$A_{1}$ and $A_{2}$, the images of $a_{1}$ and $a_{1}^{\prime}$ are
interchanged, as are the images of $b_{1}$ and $b_{1}^{\prime}$. We have now
arranged that $f_{t}(a_{1})$ and $f_{t}(a_{2})$ are in $\partial S$, so that
we can homotop $f_{t}(t)$ into $S$ fixing $\partial t$. Hence we can homotop
$f(\overline{t})$ into $\overline{S}$ while fixing $\partial\overline{t}$, as required.

\textit{Case 2:} $p_{t}(A_{1})\neq p_{t}(A_{2})$, and so $q_{t}(B_{1})\neq
q_{t}(B_{2})$.

This case is easier. If $f(a_{i})=b_{i}$ is in $\partial S$, then we leave $f$
as it is. Otherwise, we flip on $A_{i}$ to arrange that $f(a_{i})=b_{i}$.
After a finite number of such flips we have new maps $f$, $f_{t}$ which send
$\partial t$ homeomorphically onto $\partial S$ and which induce the previous
map on the fundamental groups and thus map $\pi_{1}(t)$ isomorphically to
$\pi_{1}(S)$. In particular, we can homotop $f_{t}(t)$ into $S$ fixing
$\partial t$, as required.

In conclusion, we have shown the following. We start with a given map
$f:M\rightarrow N$ such that $f(V_{0})\subset W_{0}$, $f(V_{1})\subset W_{1}$,
and $f$ is a homeomorphism on the edge spaces $V_{0}\cap V_{1}$. Suppose that
$\partial_{0}X$ has no annulus component, and also has no component
$\overline{t}$ such that $\overline{t}$ is not closed, and $\pi_{1}(\bar{t})$
has finite index in $\pi_{1}(X)$. If $\overline{t}$ is a component of
$\partial_{0}X$, then after applying flips on some of the annuli in
$\partial_{1}X$ which meet $\overline{t}$, we can homotop $f(\overline{t})$
fixing $\partial\overline{t}$, to arrange that $f$ maps $\overline{t}$ to a
component $\overline{S}$ of $\partial_{0}Y$ by a homeomorphism.

Now we examine how to inductively extend the above procedure to the remaining
components of $\partial_{0}X$. Let $\bar{r}\neq\bar{t}$ be a component of
$\partial_{0}X$. As above, after applying flips on some of the annuli in
$\partial_{1}X$ which meet $\overline{r}$, we can homotop$f(\overline{r})$
fixing $\partial\overline{r}$, to arrange that $f$ maps $\overline{r}$ to a
component $\overline{R}$ of $\partial_{0}Y$ by a homeomorphism. We claim that
$\overline{R}$ cannot equal $\overline{S}$. For if they were equal, then
$\bar{r}$ and $\overline{t}$ would be homotopic to each other in $X$. As
neither is an annulus, choosing a non-peripheral curve in $\overline{t}$
yields a $\pi_{1}$--injective annulus joining $\bar{r}$ and $\overline{t}$
which cannot be homotopic into $\partial_{1}X$ while keeping its boundary in
$\partial_{0}X$. This is an essential annulus in $(M,\partial M)$ which cannot
be homotoped into a $V_{0}$--vertex which is impossible.

We conclude that $\overline{R}$ cannot equal $\overline{S}$, so that they are
disjoint. Now there is no problem except possibly if there is an annulus
component $A$ of $\partial_{1}X$ which meets both $\bar{r}$ and $\bar{t}.$ If
we need to flip on $A$ in order to map $\overline{r}$ to $\overline{R}$, then
before the flip, the circle $A\cap\overline{r}$ must be mapped to the same
component of $\partial B$ as $A\cap\overline{t}$, which is impossible, as the
restriction of $f$ to $A$ is a homeomorphism.

Thus after a finite number of steps we can modify $f$ by a homotopy, which is
a product of flips, which preserves the decomposition, and thus carries each
group in $\partial_{1}X$ to a group in $\partial_{1}Y$ and similarly
$\partial_{0}X$ to $\partial_{0}Y$ and giving a homeomorphism $\partial X$ to
$\partial Y$. This completes the proof of Theorem \ref{thm1}.
\end{proof}

\section{An elementary proof of Johannson's Deformation
Theorem\label{applnsto3-manifolds}}

The previous section of this paper yields a new proof of Johannson's
Deformation Theorem, assuming three key ingredients, the JSJ decomposition
theorem and its enclosing property, \cite{JacoShalen} and \cite{Johannson},
and the Annulus Theorem. The literature in this area is large and contains a
confusing tangle of interlocking papers and results, some of which use the
same terms but with different definitions. After giving a brief sketch of the
history, we will show how to carve a path through the literature which yields
a complete and elementary proof of Johannson's Deformation Theorem.

We will start by discussing the JSJ decomposition of an orientable Haken
$3$--manifold $M$ with incompressible boundary, and its enclosing property. In
\cite{JacoShalen} and \cite{Johannson}, the authors used somewhat different
arguments to prove the existence and uniqueness of the JSJ decomposition, and
its enclosing property. Johannson \cite{Johannson} based his arguments on his
theory of boundary patterns, which is a way of dealing with $M$ by cutting it
into balls. Indirectly this comes down to arguing by induction on the length
$l(M)$ of the hierarchy of $M$. Jaco and Shalen \cite{JacoShalen} on the other
hand argued directly by induction on $l(M)$. In each case, the authors found
it necessary to formulate and prove a relative version of the
JSJ\ decomposition, in order that their induction step would work. In each
case, the authors showed that their decomposition was unique, in the sense
that homeomorphic $3$--manifolds had homeomorphic JSJ decompositions. The
Enclosing Property of the JSJ decomposition of $M$ is that any essential map
of an annulus or torus into $M$ can be properly homotoped into the
characteristic submanifold $V(M)$. This is not quite enough to characterise
their decomposition. That requires a more complicated statement in terms of
essential maps of Seifert pairs into $M$.

These original papers are complicated and hard to read. An elementary proof of
the existence of the JSJ\ decomposition was given by Neumann and Swarup in
\cite{NS}. Their arguments greatly simplified the subject by concentrating
only on embedded annuli and tori. This meant that they only proved a weak
version of the enclosing property, namely that any essential embedding of an
annulus or torus into $M$ can be properly homotoped into the characteristic
submanifold $V(M)$. The decomposition obtained by Neumann and Swarup in
\cite{NS} is characterised by the fact that it is a decomposition along a
minimal family of essential annuli and tori that decompose $M$ into fibred and
simple non-fibred pieces. In section 4 of chapter V of \cite{JacoShalen}, Jaco
and Shalen show that their decomposition has the same characterisation and so
these decompositions must be equal.

The enclosing property is closely related to the Annulus and Torus Theorems.
Each of these results has a long and complicated history with many different
proofs and many different versions of the statement. See \cite{Sc1} and
\cite{Sc2} for a discussion of this history. In \cite{Sc2}, Scott proved
versions of the Annulus and Torus Theorems, and showed how they implied the
enclosing property of the JSJ decomposition. The arguments in \cite{Sc2} can
be simplified by using the technology of least area surfaces \cite{FHS}. For
the Torus Theorem this has been carried out by Casson (unpublished).

Now let $M$ be an orientable Haken $3$--manifold with incompressible boundary,
and let $G=\pi_{1}(M)$. In chapter 1 of \cite{SS03}, Scott and Swarup briefly
discussed why their graph of groups decomposition $\Gamma_{1,2}(G)$ is closely
analogous to the JSJ\ decomposition of $M$, but the analogy need not be exact.
In \cite{SS05}, Scott and Swarup defined the completion $\Gamma_{1,2}^{c}(G)$
of $\Gamma_{1,2}(G)$ and explained why this is the exact analogue of the JSJ
decomposition. More precisely, if $M$ is a Haken $3$--manifold with
incompressible boundary, and $\pi_{1}(M)=G$, then $\Gamma_{1,2}^{c}(G)$ is
equal to the graph of groups decomposition $\Gamma_{M}$ of $G$ determined by
the frontier $frV(M)$ of $V(M)$. This is because the main results of
\cite{SS02} imply that the splittings of $G$ determined by $frV(M)$ are the
same as the edge splittings of $\Gamma_{1,2}(G)$, from which the result
follows easily. This fact strengthens the uniqueness of the JSJ\ decomposition
of $M$ as it implies this decomposition depends only on $G$, and not on the
homeomorphism type of $M$. This stronger uniqueness result immediately implies
that if $M$ and $N$ are orientable Haken $3$--manifolds with incompressible
boundary, and if $f:M\rightarrow N$ is a homotopy equivalence, then $f$ can be
homotoped so as to preserve their JSJ decompositions. This means that $f$ maps
$V(M)$ to $V(N)$ by a homotopy equivalence, and maps the closure of $M-V(M)$
to the closure of $N-V(N)$ by a homotopy equivalence. This is the starting
point of our proof of Johannson's Deformation Theorem in this paper, and of
the proof given by Scott and Swarup in \cite{SS02}. This stronger uniqueness
result was not known when Johannson and Jaco gave their proofs of the
Deformation Theorem.

Now we turn to the proof of Johannson's Deformation Theorem obtained from the
arguments of the preceding section. Let $M$ and $N$ be orientable Haken
$3$--manifolds with incompressible boundary, and let $f:M\rightarrow N$ be a
homotopy equivalence.

The algebraic proof of part 1) of Lemma \ref{casewhentisanannulus} can be
replaced in the manifold setting by Proposition 3.2 of \cite{NS}. The proofs
of the rest of Lemma \ref{casewhentisanannulus}, and of Lemma
\ref{casewhentisclosed}, and Proposition \ref{prop0} apply unchanged in the
manifold setting. Now Proposition \ref{prop1} implies that the map
$H_{2}(X_{t},\partial_{0}X_{t})\rightarrow H_{2}(X_{t},\partial X_{t})$ (with
$\partial X_{t}=\partial_{0}X_{t}\cup\partial_{1}X_{t}$) is zero. In
\cite{SS02}, this was a crucial step in the proof of Lemma 3.6, but the
argument there works only for manifolds. The proof of Lemma 3.6 of \cite{SS02}
on page 4995, then uses duality to conclude that the map $H_{2}(Y_{t}%
,\partial_{0}Y_{t})\rightarrow H_{2}(Y_{t},\partial Y_{t})$ is also zero.
Thus, in the manifold setting, Proposition \ref{prop7} follows immediately
from Proposition \ref{prop1}, so the proofs of Propositions \ref{prop2} to
\ref{prop7} can be omitted. The arguments in the remainder of section
\ref{section:proofofthemainresult} apply unchanged in the manifold setting. At
this point, we have shown that $f$ can be homotoped so as to map $V(M)$ to
$V(N)$ by a homotopy equivalence, and to map the closure of $M-V(M)$ to the
closure of $N-V(N)$ by a homotopy equivalence which is a homeomorphism when
restricted to the boundary. Now we can complete the proof of Johannson's
Deformation Theorem by applying Waldhausen's classical result
\cite{Waldhausen} to further homotop $f$ to arrange that $f$ maps the closure
of $M-V(M)$ to the closure of $N-V(N)$ by a homeomorphism, as required.

The arguments used on page 4995 of \cite{SS02} are as follows. There is a
duality isomorphism between $H_{2}(X_{t},\partial_{0}X_{t})$ and $H_{c}%
^{1}(X_{t},\partial_{1}X_{t})$, and between $H_{2}(X_{t},\partial X_{t})$ and
$H_{c}^{1}(X_{t})$. Thus the map $H_{c}^{1}(X_{t},\partial_{1}X_{t}%
)\rightarrow H_{c}^{1}(X_{t})$ is zero. As the homotopy equivalence
$f:X\rightarrow Y$ lifts to a proper homotopy equivalence $X_{t}$ to $Y_{t}$,
which restricts to a homeomorphism from $\partial_{1}X_{t}$ to $\partial
_{1}Y_{t}$ it follows that the map $H_{c}^{1}(Y_{t},\partial_{1}%
Y_{t})\rightarrow H_{c}^{1}(Y_{t})$ is zero, so another application of duality
yields the required result. These duality arguments are not available in the
algebraic setting of section \ref{section:proofofthemainresult}, precisely
because the pair $(X,\partial X)$ need not be a $PD3$ pair as the boundary
need not be $\pi_{1}$--injective.

Recall that $M$ is an orientable Haken $3$--manifold with incompressible
boundary, and let $V_{\partial}(M)$ denote the union of all those components
of $V(M)$ which meet $\partial M$. Thus $V_{\partial}(M)$ has the enclosing
property for essential annuli in $M$. We call $V_{\partial}(M)$ the peripheral
characteristic submanifold of $M$. One interesting point about the arguments
in this paper, is that they only involve $V_{\partial}(M)$, rather than all of
$V(M)$. It is an open question whether there is an elementary way to construct
$V_{\partial}(M)$ using the ideas of \cite{NS}. From the algebraic point of
view, the situation is clear. Recall that if $G=\pi_{1}(M)$, then
$\Gamma_{1,2}^{c}(G)$ is equal to $\Gamma_{M}$, the graph of groups structure
of $G$ determined by $frV(M)$. It is now straightforward to deduce that
$\Gamma_{1}^{c}(G)$ is equal to the graph of groups structure of $G$
determined by $frV_{\partial}(M)$. This could also be proved directly without
discussing $\Gamma_{1,2}^{c}(G)$ at all. Note that $\Gamma_{1}(G)$ is defined,
and shown to exist in \cite{SS03}, and its completion $\Gamma_{1}^{c}(G)$ is
defined in \cite{SS05}.

\end{document}